\documentclass{amsart}
\usepackage{amsmath}
\usepackage{latexsym}
\usepackage{color}
\usepackage{amscd}
\usepackage[all]{xy}
\usepackage{enumerate}
\usepackage{soul}
\usepackage{graphicx}
\usepackage{comment}
\newtheorem{thm}{Theorem}[section]
\newtheorem{lem}[thm]{Lemma}
\newtheorem{cor}[thm]{Corollary}
\newtheorem{prop}[thm]{Proposition}

\newtheorem{ques}[thm]{Question}

\theoremstyle{definition}

\def\M{\mathcal {M}_{g,p}}
\def\MO{\mathcal {M}_{g,0}}
\def\EM{\mathcal {M}^\pm_{g,p}}
\def\EMO{\mathcal {M}^\pm_{g,0}}
\def\PM{\mathcal {PM}_{g,p}}
\def\S{\Sigma_{g,p}}
\def\Sym{\mathrm{Sym}_p}

\begin{document}

\title[On minimal generating sets for the mapping class group of a punctured surface]
{On minimal generating sets for the mapping class group of a punctured surface}

\author[N. Monden]{Naoyuki Monden}
\address{Department of Mathematics, Faculty of Science, Okayama University, Okayama 700-8530, Japan}
\email{n-monden@okayama-u.ac.jp}

\begin{abstract}
Let $\S$ be a oriented connected surface of genus $g$ with $p$ punctures. We denote by $\M$ and $\EM$ the mapping class group and the extended mapping class group of $\S$, respectively. 
In this paper, we show that $\M$ and $\EM$ are generated by two element for $g\geq 3$ and $p\geq 0$. 
\end{abstract}

\maketitle

\section{Introduction}
Let $\S$ be a surface obtained from a compact oriented connected closed surface $\Sigma_g$ of genus $g$ by removing $p$ points. 
We say that $\S$ has $p$ \textit{punctures}. 
We denote by $\M$ the \textit{mapping class group} of $\S$, i.e., the group of homotopy classes of orientation-preserving diffeomorphisms which preserve the set of punctures. 
Let $\EM$ be the extended mapping class group of $\S$, i.e., the group of isotopy class of all (including orientation-reversing) diffeomorphisms which preserve the set of punctures. 
The groups $\M$ and $\EM$ are related by the following exact sequence:
\begin{align*}
1 \to \M \to \EM \to \mathbb{Z}/2\mathbb{Z} \to 1. 
\end{align*}

Since the mapping class groups $\M$ and $\EM$ are not cyclic, any generating sets for $\M$ and $\EM$ consisting must have cardinality at least two. 
In this paper, we prove that $\M$ and $\EM$ are generated by two for $g\geq 3$. 
\begin{thm}\label{thm:1}
For $g \geq 3$ and $p \geq 0$, $\M$ is generated by two elements. 
\end{thm}
\begin{thm}\label{thm:10}
For $g \geq 3$ and $p \geq 0$, $\EM$ is generated by two elements. 
\end{thm}


We give a brief history of finding generating sets for $\M$. 
Dehn \cite{De} gave a generating set for $\MO$ consisting of $2g(g-1)$ Dehn twists. 
After that, Lickorish \cite{Li} showed that $3g-1$ Dehn twists generate $\MO$, and Humphries \cite{Hu} reduced the number to $2g+1$. 
In particular, he proved that in fact the number $2g+1$ is minimal for $g\geq 2$; i.e. $\MO$ cannot be generated by $2g$ (or less) Dehn twists. 
Johnson \cite{Jo2} proved that the $2g-1$ Dehn twists also generate $\mathcal{M}_{g,1}$. 
Labruere and Paris \cite{LP} showed that $\M$ is generated by $2g+2$ Dehn twists and $p-1$ half twists and gave a finite presentation of $\M$.

By using elements other than Dehn twists and half twists, we can obtain smaller generating sets for $\MO$. 
Lu \cite{Lu} showed that $\MO$ is generated by three elements. 
After that, Wajnryb \cite{Wa} found the smallest possible generating set for $\mathcal{M}_{g,p}$ consisting of two elements for $g\geq 1$ and $p=0,1$. 
Korkmaz \cite{Ko1} proved that one of these two generators may be taken to be a Dehn twist. 
Moreover, he also showed that $\EMO$ is generated by two elements for $g\geq 1$ and $p=0,1$. 
In the case of $p\geq 2$, Kassabov \cite{Ka} showed that $\M$ is generated by four involutions if $g > 7$ or $g = 7$ and $p$ is even, five involutions if $g > 5$ or $g = 5$ and $p$ is even, and 6 involutions if $g > 3$ or $g = 3$ and $p$ is even. 
The assumption that $p$ is even was removed by the author \cite{Mo2} for $g=7$ and $g=5$. 
Based on the idea of \cite{Ka}, the author \cite{Mo1} constructed a generating set for $\mathcal{M}_{g,p}$ consisting of three elements for $g\geq 1$.

Although $\M$ is finitely generated by Dehn twists and half twists, we don't know how the generators given in \cite{Ka}, \cite{Mo1} and \cite{Mo2} are factorized into them. 
On the other hand, the generators appeared in Theorem~\ref{thm:1} has explicit factorizations.

The outline of the paper is as follows. 
In Section~\ref{section2}, we present some fundamental results on mapping class groups. 
The proofs of Theorems~\ref{thm:1} and~\ref{thm:10} 
are given in Sections~\ref{section3} and~\ref{section4}, respectively. 
The last section contains further remarks on various generating sets for $\M$ and $\EM$. 

\vspace{0.1in}
\noindent \textit{Acknowledgements.} 
The author would like to thank S. Hirose for telling him many results on generators for the mapping class groups and M. Pamuk for pointing out the reference \cite{APY}. 
The author was supported by Grant-in-Aid for Scientific Research (C) (No. 20K03613), Japan Society for the Promotion of Science.

\section{Preliminaries}\label{section2}
%
Let $t_c$ be the right-handed Dehn twist about a simple closed curve $c$ on $\Sigma_{g,p}$, that is, the isotopy class of the diffeomorphism obtained by cutting $\S$ along $c$, twisting one of the side by $2\pi$ to the right and gluing two sides of $c$ back to each other. 
Let $\sigma_\ell$ be the right-handed half twist along the simple arc $\ell$ joining two punctures $x_i$ and $x_j$ of $\S$, that is, the isotopy class of the diffeomorphism supported in a small regular neighborhood of $N(\ell \cup x_i \cup x_j)$ which leaves $\ell$ invariant and interchanges $x_i$, $x_j$, such that $\sigma_\ell^2$ is the right handed Dehn twist about $\partial N(\ell \cup x_i \cup x_j)$. 
For $f_1,f_2 \in \M$, the composition $f_2f_1$ means that $f_1$ is applied first. 
We recall the basic facts about $\M$. 
More details can be found in \cite{FM}. 
\begin{itemize}
\item If a simple closed curve $c$ on $\S$ is homotopic to a puncture, then $t_c=1$. 
\item For a self-diffeomorphism $f$ of $\S$ and a simple closed curve $c$ on $\S$, we have the relation $t_{f(c)} = ft_c^{\varepsilon}f^{-1}$ in $\M$, where $\varepsilon = \pm 1$ depending on whether $f$ is orientation-preserving or orientation-reversing. 
Similarly, for a simple arc $\ell$ on $\S$, we have the relation $\sigma_{f(\ell)} = f\sigma_\ell^{\varepsilon}f^{-1}$. 
\item For two simple closed curves $a,a'$ on $\S$ which are disjoint from each other, the \textit{commutative relation} $t_a t_{a'} = t_{a'} t_a$ holds. 
Similarly, for two simple arcs $\ell,\ell'$ on $\S$ which are disjoint from each other, $\sigma_\ell \sigma_{\ell'} = \sigma_{\ell'} \sigma_{\ell}$. 
Moreover, if $a$ is disjoint from $\ell$, then $t_a \sigma_\ell = \sigma_\ell t_a$. 
\item  Let $x,y,z$ be the interior curves on a subsurface of genus $0$ with four boundary curves $a,b,c,d$ in $\S$ as in Figure~\ref{lanterncurves}. 
Then, the \textit{lantern relation} $t_d t_c t_b t_a = t_z t_y t_x$ holds. 
The lantern relation was discovered by Dehn \cite{De} and rediscovered by Johnson \cite{Jo1}. 
\begin{figure}[hbt]
  \centering
       \includegraphics[scale=.25]{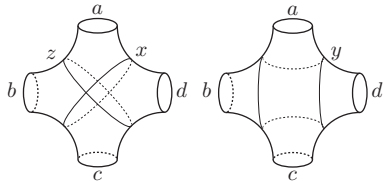}
       \caption{The simple closed curves $a,b,c,d,x,y,z$ on $\S$ for the lantern relation.}
       \label{lanterncurves}
  \end{figure}
\end{itemize}

Here, we assume that the surface of this paper is the $yz$-plane, and $\S$ is embedded in $\mathbb{R}^3$ as in Figure~\ref{curvesarcs} such that it is invariant under the reflection $R$ across the $yz$-plane. 
Therefore, the set of $p$ punctures $x_1,x_2,\ldots,x_p$ is also in the $yz$-plane as in Figure~\ref{curvesarcs}, and in particular, $R(x_k)=x_k$ for $k=1,2,\ldots,p$. 
\begin{figure}[hbt]
  \centering
       \includegraphics[scale=.45]{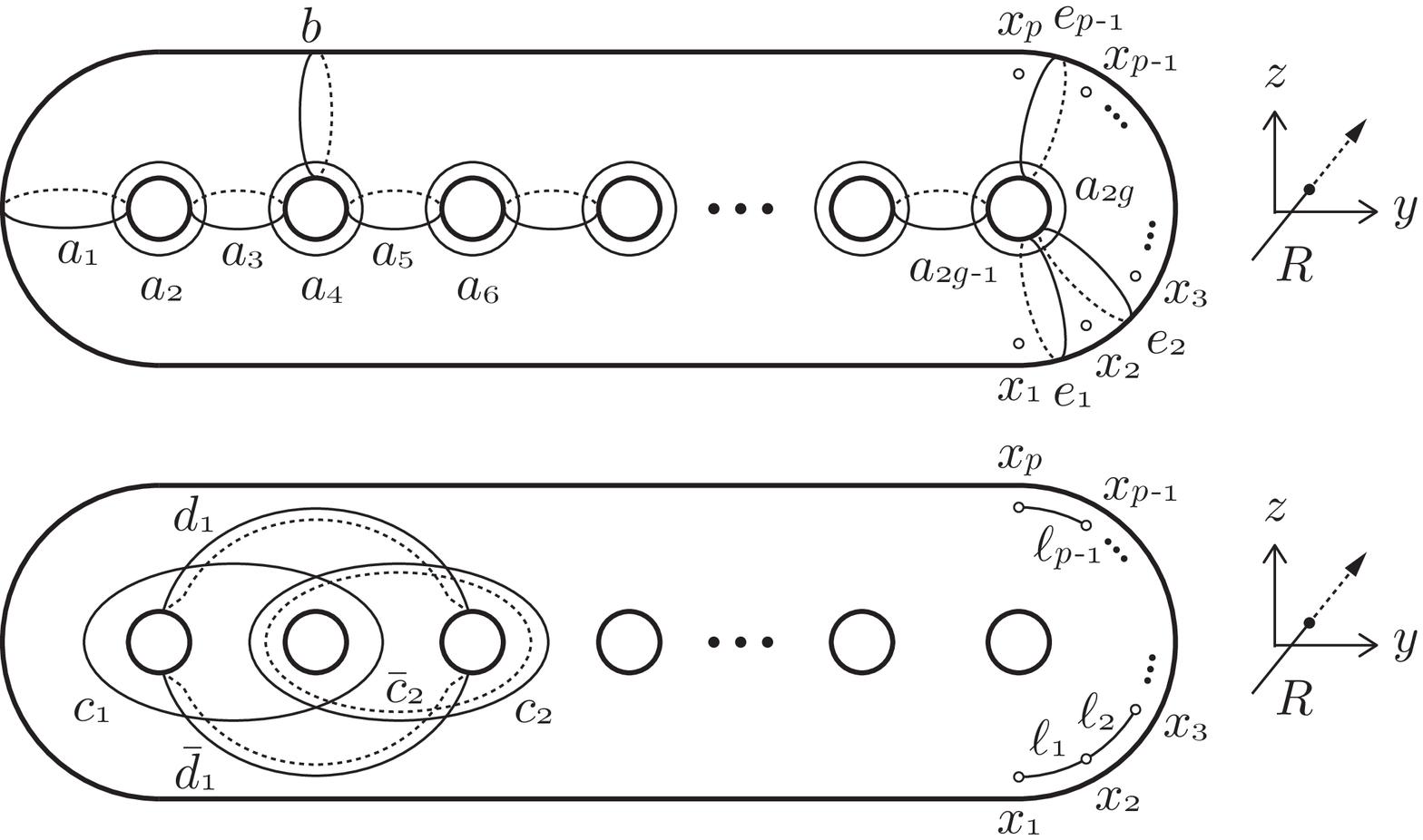}
       \caption{The reflection $R$ of $\S$, the simple closed curves $a_i$
       , $b$, $c_1,c_2,\overline{c}_2$, $d_1,\overline{d}_1$, $e_k$
        and the arc $\ell_k$ on $\S$.}
       \label{curvesarcs}
  \end{figure}

We define the simple closed curves $a_i$ ($i=1,2,\ldots,2g$), $b$, $c_1,c_2,\overline{c}_2$, $d_1,\overline{d}_1$, $e_k$ ($k=1,2,\ldots,p-1$) and the arc $\ell_k$ on $\S$ as shown in Figure~\ref{curvesarcs}. 
For simplicity, we denote the right-handed Dehn twists about $a_i, b, c_j, \overline{c}_j, d_j, \overline{d}_j, e_k$ by capital letters $A_i, B, C_j, \overline{C}_j, D_j, \overline{D}_j, E_k$ and denote the right-handed half twist about $\ell_k$ by $\sigma_k$.

The following facts in combination with the relations $t_{f(c)}=ft_c^\varepsilon f^{-1}$ and $\sigma_{f(\ell)}=f\sigma_\ell^\varepsilon f^{-1}$ are used repeatedly to prove Theorems~\ref{thm:1} and~\ref{thm:10}. 
From the definition of the reflection $R$, we see that $R(a_i)=a_i$ for $i=1,2,\ldots,2g$, $R(b)=b$, $R(e_k)=e_k$ and $R(\ell_k) = \ell_k$ for $k=1,2,\ldots,p-1$. 
Let 
\[F:=\sigma_{p-1} \cdots \sigma_2 \sigma_1 E_{p-1} A_{2g} \cdots A_{2} A_1. \]
Then, it is easy check that 
\begin{align*}
&F^{-1}(a_i) = a_{i+1}, &(RF)^{-1}(a_i) = a_{i+1}, \\
&F^{-1}(a_{2g}) = e_{p-1},& (RF)^{-1}(a_{2g}) = e_{p-1}, \\
&F^{-1}(\ell_k) = \ell_{k+1}, & (RF)^{-1}(\ell_k) = \ell_{k+1}
\end{align*}
for $i=1,2,\ldots,2g-1$ and $k=1,2,\ldots,p-3$. 
Note that $F^{-1}(\ell_{p-2}) \neq \ell_{p-1}$ since $\ell_{p-1}$ intersects $e_{p-1}$ at once. 
Similarly, we see that 
\begin{align*}
&F^{-1}(b) = c_1, &(RF)^{-1}(b) = c_1, \\
&F^{-1}(c_1) = d_1,& (RF)^{-1}(c_1) = \overline{d}_1, \\
&F^{-1}(d_1) = c_2,& (RF)^{-1}(\overline{d}_1) = \overline{c}_2. 
\end{align*}


In this paper, we employ a slightly different generating set of $\M$ from that given in \cite{LP}, but they are essentially same. 
For that reason, we give a proof, but the argument is similar to that \cite{LP}. 
\begin{prop}
\label{prop:4}
For $g\geq 1$ and $p\geq 2$, $\M$ is generated by $A_1,A_2,\ldots,A_{2g}$, $B$, $E_{p-1}$ and $\sigma_1,\sigma_2,\ldots,\sigma_{p-1}$. 
\end{prop}
\begin{proof}
By $\PM$ we will denote the subgroup of $\M$ which fixes the punctures pointwise. 
We call $\PM$ the \textit{pure mapping class group} of $\S$. 
It is well-known that $\PM$ is generated by  $2g+p$ Dehn twists $A_1,A_2,\ldots,A_{2g},B$ and $E_1,E_2,\ldots,E_{p-1}$ (see, for example \cite{Ge} and also Section 4.4.4 in \cite{FM}). 
The action of $\M$ on the punctures of $\S$ gives the following short exact sequence:
\begin{align*}
1 \to \PM \to \M \to \Sym \to 1, 
\end{align*}
where $\Sym$ is the permutation group on the $p$ punctures and the last projection is given by the restriction of a diffeomorphism to its action on the punctures. 
Since the image of the half twist $\sigma_1,\sigma_2,\ldots,\sigma_{p-1}$ in $\M$ under the last projection are generators of $\Sym$, we see that 
$\M$ is generated by $A_1,A_2,\ldots,A_{2g}$, $B$, $E_{1},E_2,\ldots,E_{p-1}$ and $\sigma_1,\sigma_2,\ldots,\sigma_{p-1}$. 
Therefore, we only need to show that $E_1,E_2,\ldots,E_{p-2}$ are generated by $A_1,A_2,\ldots,A_{2g},B,E_{p-1}$ and $\sigma_1,\sigma_2,\ldots,\sigma_{p-1}$.

Let $\delta$ be the simple closed curve as in Figure~\ref{curves} which separates $\S$ into two components: the first one, denoted by $\Sigma$, is a surface of genus $g$ with one boundary component and no punctures. 
The second one is a disk with $p$ punctures. 
Let $E_p$ be the Dehn twist about the simple closed curve $e_p$ on $\Sigma$ as in Figure~\ref{curves1}. 
When we consider the mapping class group $\mathcal{M}(\Sigma)$ of $\Sigma$, i.e., the group of isotopy classes of orientation-preserving diffeomorphisms of $\S$ which restrict to the identity on $\S-\mathrm{Int}\Sigma$, $\mathcal{M}(\Sigma)$ is generated by the Dehn twists $A_1,A_2,\ldots,A_{2g}$ and $B$ (see \cite{Jo2}). 
Therefore, $E_p$ is generated by $A_1,A_2,\ldots,A_{2g},B$. 
\begin{figure}[hbt]
  \centering
       \includegraphics[scale=.45]{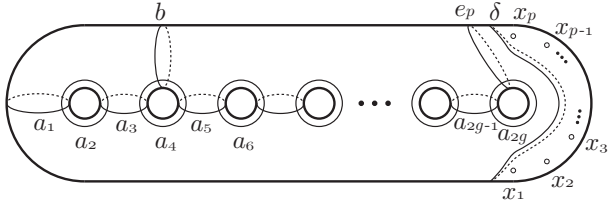}
       \caption{The simple closed curves $\delta$ and $e_p$ on $\S$.}
       \label{curves1}
  \end{figure}

Let $\delta_{p-1}$ and $\delta_{p}$ be the simple closed curves which are homotopic to the punctures $x_{p-1}$ and $x_{p}$, respectively, and let $\delta_{p-1,p}$ be the separating simple closed curve as in Figure~\ref{curves2}. 
We denote by $\Delta_{p-1},\Delta_{p},\Delta_{p-1,p}$ be the Dehn twists about $\delta_{p-1},\delta_{p},\delta_{p-1,p}$, respectively. 
Then, the lantern relation 
\begin{align*}
E_{p-2} E_p \Delta_{p} \Delta_{p-1} = E_{p-1} E_{p-1}' \Delta_{p-1,p}
\end{align*}
holds, where $E_{p-1}'$ be the Dehn twist about the simple closed curve $\sigma_{p-1}(e_{p-1})$. 
\begin{figure}[hbt]
  \centering
       \includegraphics[scale=.45]{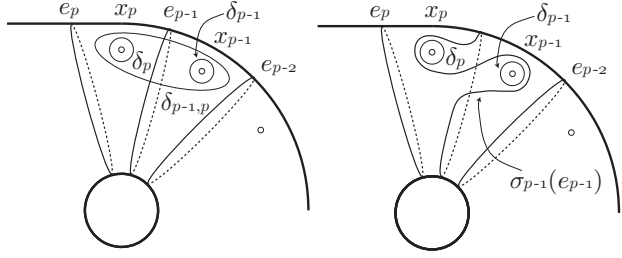}
       \caption{The simple closed curves $\delta_{p-2},\delta_{p-1},\delta_{p-2,p-1}$ and $\sigma_{p-1}(e_{p-1})$ on $\S$.}
       \label{curves2}
  \end{figure}
Since $\Delta_{p-1}=\Delta_{p}=1$, $\Delta_{p-1,p}=\sigma_{p-1}^2$ and $E_{p-1}'= t_{\sigma_{p-1}(e_{p-1})} = \sigma_{p-1} E_{p-1} \sigma_{p-1}^{-1}$, we rewrite this relation as
\begin{align*}
E_{p-2} = E_{p-1} \sigma_{p-1} E_{p-1} \sigma_{p-1}^{-1} \sigma_{p-1}^2 E_p^{-1} . 
\end{align*}
Similarly, we obtain $E_{j} = E_{j+1} \sigma_{j+1} E_{j+1} \sigma_{j+1}^{-1} \sigma_{j+1}^2 E_{j+2}^{-1}$ for $j=1,2,\ldots,p-2$. 
Inductively, we see that $E_{j}$ is generated by $A_1,A_2,\ldots,A_{2g},B,E_{p-1}$ and $\sigma_1,\sigma_2,\ldots,\sigma_{p-1}$ for all $j$ since $E_p$ is generated by $A_1,A_2,\ldots,A_{2g},B$, and the theorem follows. 
\end{proof}

\section{Proofs of Theorem~\ref{thm:1} }\label{section3}

The next lemma is needed to prove Theorem~\ref{thm:1}. 
\begin{lem}\label{lem:1}
Set $F:=\sigma_{p-1} \cdots \sigma_2 \sigma_1 E_{p-1} A_{2g} \cdots A_{2} A_1$. 
Let $G_1$ be a subgroup of $\M$ generated by $A_1 A_2^{-1}$, $A_1 B^{-1}$ and $F$. 
Then, the Dehn twists $A_1,A_2,\ldots,A_{2g},B$ and $E_{p-1}$ are in $G_1$. 
\end{lem}
\begin{proof}
By $F^{-1} (a_1,b) = (a_2,c_1)$, $F^{-1} (a_2,c_1) = (a_3, d_1)$ and $F^{-1} (a_3,d_1) = (a_4, c_2)$, 
we have $A_2 C_1^{-1} = F^{-1} A_1 B^{-1} F$, $A_3 D_1^{-1} = F^{-2} A_1 B^{-1} F^2$ and $A_4 C_2^{-1} = F^{-3} A_1 B^{-1} F^3$. 
Therefore, 
\begin{align}
A_2 C_1^{-1}, \ A_3 D_1^{-1}, \ A_4 C_2^{-1} \in G_1. \label{in1a}
\end{align}

Here, we denote by $E$ and $B_2$ the Dehn twists about simple closed curves $e$ and $b_2$ as in Figure~\ref{curves}. 
Then, the lantern relation 
\begin{align*}
B_2 A_5 A_3 A_1 = D_1 E B
\end{align*}
holds, and since $a_1,a_3,a_5,b_2$ are disjoint from each other and $d_1,e,b$ we rewrite this relation as 
\begin{align*}
A_5 = (D_1A_3^{-1}) (EB_2^{-1}) (BA_1^{-1}). 
\end{align*}
If $EB_2^{-1} \in G_1$, then we get $A_5 \in G_1$ by $A_1B^{-1} \in G_1$ (by the assumption) and $A_3D_1^{-1} \in G_1$ (by (\ref{in1a})). 
Moreover, since $F^{-i}(a_1)=a_{i+1}$ for $i=1,2,\ldots,2g-1$ and $F^{-1}(a_{2g}) = e_{p-1}$, we obtain $A_{5+j} = F^{-j}A_5F^j$ for $j=-4,-3,\ldots,2g-5$ and $E_{p-1} = F^{-2g+4}A_5F^{2g-4}$. 
Therefore, if $A_5 \in G_1$, then $A_1,A_2,\ldots,A_{2g},E_{p-1} \in G_1$ and $B = BA_1^{-1} \cdot A_1 \in G_1$. 
From the above argument, it suffices to show that $EB_2^{-1} \in G_1$. 
\begin{figure}[hbt]
  \centering
       \includegraphics[scale=.45]{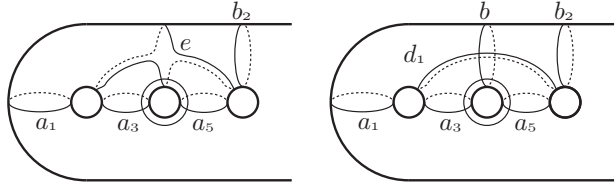}
       \caption{The simple closed curves $a_1,a_3,a_5,b_2,b,d_1,e$ on $\S$.}
       \label{curves}
  \end{figure}

We show that $EB_2^{-1} \in G_1$. 
Since $F^{-1}(a_i)=a_{i+1}$ for for $i=1,2,\ldots,2g-1$, we see that $A_{j+1}A_{j+2}^{-1} = F^{-j} A_1A_2^{-1} F^j$. 
This gives $A_1A_{j+2}^{-1} = A_1A_2^{-1} \cdot A_2A_3^{-1} \cdots A_{j+1}A_{j+2}^{-1}$. 
In particular, 
\begin{align}
A_1 A_4^{-1}, \ A_1 A_5^{-1}, \ A_1A_6^{-1} \in G_1. \label{in2a}
\end{align}
Similarly, we have $A_2A_5^{-1} = F^{-1} A_1A_4^{-1} F$. 
Since $a_1,a_2$ are disjoint from $a_5$, we obtain 
\begin{align}
A_1^{-1}A_2 = (A_1A_5^{-1})^{-1} (A_2A_5^{-1}) \in G. \label{in3a}
\end{align}
It is easy to check that 
\begin{align*}
A_6 A_1^{-1} \cdot A_5 A_1^{-1} \cdot A_4 C_2^{-1} (a_1,b) = (a_1,b_2). 
\end{align*}
Therefore, by (\ref{in1a}) and (\ref{in2a}), we obtain 
\begin{align}
A_1B_2^{-1} = (A_6 A_1^{-1} \cdot A_5 A_1^{-1} \cdot A_4 C_2^{-1}) A_1 B^{-1} (A_6 A_1^{-1} \cdot A_5 A_1^{-1} \cdot A_4 C_2^{-1})^{-1} \in G_1. \label{in4a}
\end{align}
Since $a_1$ is disjoint from $b_2$, we have 
\begin{align}
B_2^{-1}A_1 \in G. \label{in8a}
\end{align}
Using (\ref{in2a}) and (\ref{in4a}), 
\begin{align}
A_5 B_2^{-1} = A_5 A_1^{-1} \cdot A_1 B_2^{-1} \in G_1. \label{in5a}
\end{align}
By (\ref{in3a}), we get
\begin{align}
B_2^{-1}A_2 = B_2^{-1}A_1 \cdot A_1^{-1}A_2 \in G_1. \label{in6a}
\end{align}
Since $a_2$ is disjoint from $c_1$, we have $A_2^{-1}C_1 \in G_1$ by (\ref{in1a}). 
From this, (\ref{in2a}) and (\ref{in3a}), we obtain
\begin{align}
A_4^{-1}C_1 = A_4^{-1}A_1 \cdot A_1^{-1}A_2 \cdot A_2^{-1}C_1 \in G_1. \label{in7a}
\end{align}
Since it is easily seen that 
\begin{align*}
B_2^{-1}A_2 \cdot B_2^{-1}A_1 \cdot A_4^{-1}C_1 (a_5,b_2) = (e,b_2),
\end{align*}
by 
(\ref{in8a}), (\ref{in5a}), (\ref{in6a}) and (\ref{in7a}) we have 
\begin{align*}
EB_2^{-1} = (B_2^{-1}A_2 \cdot B_2^{-1}A_1 \cdot A_4^{-1}C_1) A_5B_2^{-1} (B_2^{-1}A_2 \cdot B_2^{-1}A_1 \cdot A_4^{-1}C_1)^{-1} \in G_1. 
\end{align*}
which completes the proof. 
\end{proof}

We prove Theorem~\ref{thm:1} for $p\geq 7$. 
\begin{thm}\label{thm:2}
Set $F:=\sigma_{p-1} \cdots \sigma_2 \sigma_1 E_{p-1} A_{2g} \cdots A_{2} A_1$. 
If $g \geq 3$ and $p \geq 7$, then $\M$ is generated by $A_2 B^{-1} \sigma_2$ and $F$. 
\end{thm}
\begin{proof}
Let $G_2$ be the subgroup of $\M$ generated by $A_2 B^{-1} \sigma_2$ and $F$. 
We prove that $G_2=\M$ using Proposition~\ref{prop:4}, that is, $A_1,A_2,\ldots,A_{2g},B,E_{p-1}$ and $\sigma_1$, $\sigma_2,\ldots,\sigma_{p-1}$ are in $G_2$. 
By Lemma~\ref{lem:1}, it is sufficient to show that $A_1A_2^{-1}$, $A_1B^{-1}$ and $\sigma_1,\sigma_2,\ldots,\sigma_{p-1}$ are in $G_2$.

Since $F^{-3} (a_2,b,\ell_2) = (a_5, c_2,\ell_5)$ by the assumption that $p\geq 7$, we obtain $A_5 C_2^{-1} \sigma_5 = F^{-3} A_2 B^{-1} \sigma_2 F^3 \in G_2$. 
It is easily seen that 
\begin{align*}
A_5 C_2^{-1} \sigma_5 \cdot A_2 B^{-1} \sigma_2 (a_5,c_2,\ell_5) = (a_5,b,\ell_5), 
\end{align*}
and we have 
\begin{align}
A_5B^{-1}\sigma_5 = (A_5 C_2^{-1} \sigma_5 \cdot A_2 B^{-1} \sigma_2) A_5 C_2^{-1} \sigma_5 (A_5 C_2^{-1} \sigma_5 \cdot A_2 B^{-1} \sigma_2)^{-1} \in G_2. \label{in13}
\end{align}
In particular, since $\ell_5$ is disjoint from $b,a_5,c_2$, we have 
\begin{align}
BC_2^{-1} = \sigma_5^{-1} B A_5^{-1} \cdot A_5 C_2^{-1} \sigma_5 \in G_2. \label{in14d}
\end{align}
This gives 
\begin{align}
A_2C_2^{-1}\sigma_2 = A_2 B^{-1} \sigma_2 \cdot B C_2^{-1} \in G_2. \label{in14}
\end{align}
By (\ref{in13}), we have $A_2 A_5^{-1} \sigma_2 \sigma_5^{-1} = A_2 B^{-1} \sigma_2 \cdot \sigma_5^{-1} B A_5^{-1}$ in $G_2$. 
Since $F(a_2,a_5,\ell_2,\ell_5) = (a_1,a_4,\ell_1,\ell_4)$, we get 
\begin{align}
A_1 A_4^{-1} \sigma_1 \sigma_4^{-1} = F A_2 A_5^{-1} \sigma_2 \sigma_5^{-1} F^{-1} \in G_2. \label{in15}
\end{align}
Here, it is immediate that 
\begin{align*}
A_2 C_2^{-1} \sigma_2 \cdot A_1 A_4^{-1} \sigma_1 \sigma_4^{-1} (a_2,c_2,\ell_2) = (a_1,c_2,\ell_1). 
\end{align*}
This gives that
\begin{align}
A_1C_2^{-1}\sigma_1 = (A_2 C_2^{-1} \sigma_2 \cdot A_1 A_4^{-1} \sigma_1 \sigma_4^{-1}) A_2 C_2^{-1} \sigma_2 (A_2 C_2^{-1} \sigma_2 \cdot A_1 A_4^{-1} \sigma_1 \sigma_4^{-1})^{-1} \in G_2. \label{in16}
\end{align}
Using this and (\ref{in14d}), we get
\begin{align}
A_1B^{-1}\sigma_1 = A_1 C_2^{-1} \sigma_1 \cdot C_2 B^{-1} \in G_2. \label{in17}
\end{align}
Since we see at once that 
\begin{align*}
A_1B^{-1}\sigma_1 \cdot A_1 A_4^{-1} \sigma_1 \sigma_4^{-1} (a_1,b,\ell_1) = (a_1,a_4,\ell_1), 
\end{align*}
by (\ref{in15}) we have
\begin{align*}
A_1A_4^{-1}\sigma_1 = (A_1B^{-1}\sigma_1 \cdot A_1 A_4^{-1} \sigma_1 \sigma_4^{-1}) A_1B^{-1}\sigma_1 (A_1B^{-1}\sigma_1 \cdot A_1 A_4^{-1} \sigma_1 \sigma_4^{-1})^{-1} \in G_2. 
\end{align*}
Since $a_1,a_4,\ell_1,\ell_4$ are disjoint from each other, by (\ref{in15}) we obtain 
\begin{align*}
\sigma_4 = (A_1 A_4^{-1} \sigma_1 \sigma_4^{-1})^{-1} A_1A_4^{-1}\sigma_1 \in G_2. 
\end{align*}
This gives $\sigma_1,\sigma_2$ in $G$ since $F^{i}(\ell_4)=\ell_{4-i}$ for $i=1,2,3$. 
By (\ref{in14}), (\ref{in16}) and (\ref{in17}), we have 
\begin{align*}
A_2C_2^{-1}, \ A_1C_2^{-1}, \ A_1B^{-1} \in G_2, 
\end{align*}
and therefore 
\begin{align*}
A_1A_2^{-1} = A_1C_2^{-1} \cdot (A_2C_2^{-1})^{-1} \in G_2. 
\end{align*}
Therefore, we see that $A_1,A_2,\ldots,A_{2g},B,E_{p-1}$ are in $G_2$ from Lemma~\ref{lem:1}.

Finally, we show that $\sigma_j \in G_2$ for $j=1,2,\ldots,p-1$. 
From $A_2B^{-1}\sigma_2$ and the definition of $F$, we see that $\sigma_2$ and $\sigma_{p-1}\cdots \sigma_2 \sigma_1$ are in $G_2$. 
Since $(\sigma_{p-1}\cdots \sigma_2 \sigma_1)^{-i}(\ell_2) = \ell_{2+i}$ for $i=-1,0,1,\ldots,p-2$, we have $\sigma_j \in G_2$ for $j=1,2,\ldots,p-1$.

By Lemma~\ref{lem:1} and Proposition~\ref{prop:4}, we see that $G_2 =\M$, and the proof is complete. 
\end{proof}


We show Theorem~\ref{thm:1} for $0\leq p\leq 6$. 
\begin{cor}\label{cor:1}
For $g \geq 3$ and $p \geq 0$, $\M$ is generated by two elements. 
\end{cor}
\begin{proof}
There is a natural map $\Sigma_{g,p} \to \Sigma_{g,p-1}$ where the puncture $x_p$ is ``forgotten", and this map induces a surjective homomorphism $\mathcal{M}_{g,p} \to \mathcal{M}_{g,p-1}$. 
Using Theorem~\ref{thm:2} and this surjective homomorphism we see that $\mathcal{M}$ is also generated by two elements for $g\geq 3$ and $0\leq p\leq 6$. 
This finishes the proof. 
\end{proof}

\section{Proof of Theorem~\ref{thm:10}}\label{section4}
The next lemma is needed to prove Theorem~\ref{thm:10}. 
\begin{lem}\label{lem:10}
Set $F:=\sigma_{p-1} \cdots \sigma_2 \sigma_1 E_{p-1} A_{2g} \cdots A_{2} A_1$. 
Let $G_3$ be a subgroup of $\EM$ generated by $A_1 A_2^{-1}$, $A_1 B^{-1}$ and $RF$. 
Then, the Dehn twists $A_1,A_2,\ldots,A_{2g},B$ and $E_{p-1}$ are in $G_3$. 
\end{lem}
\begin{proof}
It is easy to check that $(RF)^{-1} (a_1,b) = (a_2,c_1)$, $(RF)^{-1} (a_2,c_1) = (a_3, \overline{d}_1)$ and $(RF)^{-1} (a_3,\overline{d}_1) = (a_4, \overline{c}_2)$. 
This gives $A_2^{-1} C_1 = (RF)^{-1} A_1 B^{-1} (RF)$, $A_3 \overline{D}_1^{-1} = (RF)^{-2} A_1 B^{-1} (RF)^2$ and $A_4^{-1} \overline{C}_2 = (RF)^{-3} A_1 B^{-1} (RF)^3$. 
Therefore, 
\begin{align}
A_2^{-1} C_1, \ A_3 \overline{D}_1^{-1}, \ A_4^{-1} \overline{C}_2 \in G_3. \label{in1b}
\end{align}

Here, we denote by $\overline{E}$, $\overline{B}$ and $\overline{B}_2$ the Dehn twists about simple closed curves $\overline{e}$, $\overline{b}$ and $\overline{b}_2$ as in Figure~\ref{curves}, respectively. 
Then, the lantern relation 
\begin{align*}
\overline{B}_2 A_5 A_3 A_1 = \overline{B}_{} \ \overline{E} \ \overline{D}_1 
\end{align*}
holds, and since $a_1,a_3,a_5,\overline{b}_2$ are disjoint from each other and $\overline{d}_1,\overline{e},\overline{b}$, we rewrite this relation as 
\begin{align*}
A_5 = (\overline{B} A_1^{-1}) (\overline{E}_{} \ \overline{B}_2^{-1}) (\overline{D}_1 A_3^{-1}). 
\end{align*}
If $\overline{B} A_1^{-1}, \ \overline{E}_{} \ \overline{B}_2^{-1} \in G_3$, then we get $A_5 \in G_3$ by $A_3 \overline{D}_1^{-1} \in G_3$ (by (\ref{in1b})). 
Moreover, since $(RF)^{-i}(a_1)=a_{i+1}$ for $i=1,2,\ldots,2g-1$ and $(RF)^{-1}(a_{2g}) = e_{p-1}$, we obtain $A_{5+j}^{\varepsilon} = (RF)^{-j}A_5(RF)^j$ for $j=-4,-3,\ldots,2g-5$ and $E_{p-1} = (RF)^{-2g+4}A_5(RF)^{2g-4}$, where $\varepsilon=-1$ if $j$ is odd and $\varepsilon=1$ if $j$ is even. 
Therefore, if $A_5 \in G_3$, then $A_1,A_2,\ldots,A_{2g},E_{p-1} \in G_3$ and $B = BA_1^{-1} \cdot A_1 \in G_3$. 
From the above argument, it suffices to show that $\overline{B} A_1^{-1}, \ \overline{E}_{} \ \overline{B}_2^{-1} \in G_3$. 
\begin{figure}[hbt]
  \centering
       \includegraphics[scale=.45]{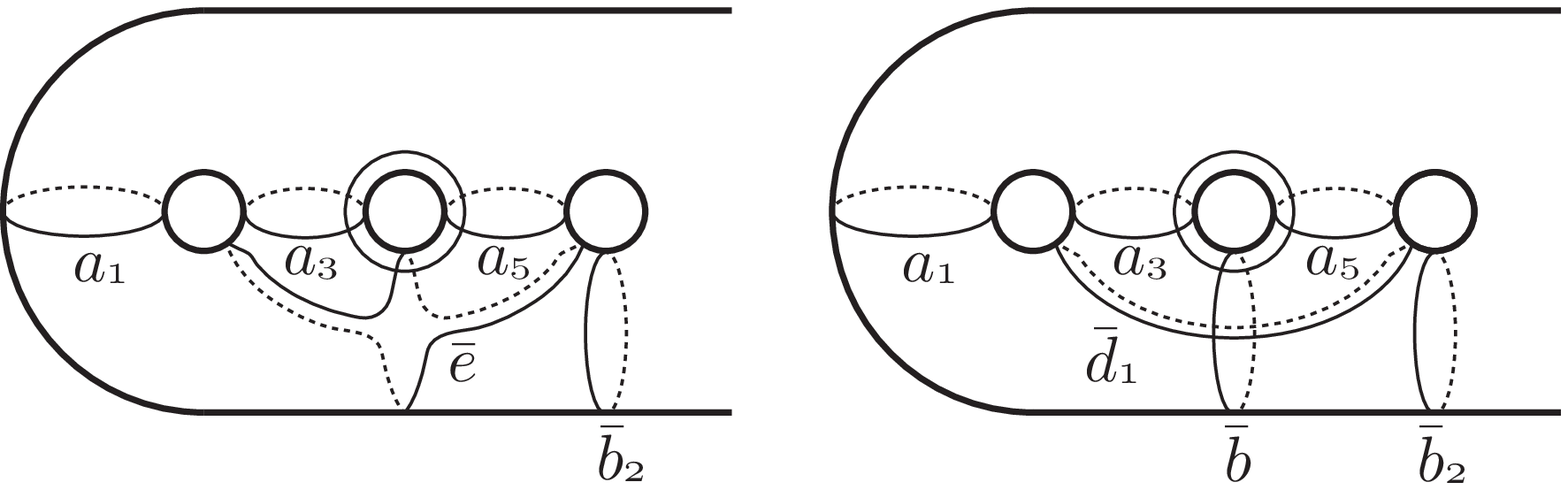}
       \caption{The simple closed curves $a_1,a_3,a_5,\overline{b}_2,\overline{b},\overline{d}_1,\overline{e}$ on $\S$.}
       \label{curves}
  \end{figure}

First, show that $\overline{B} A_1^{-1} \in G_3$. 
Since $(RF)^{-1}(a_i)=a_{i+1}$ for for $i=1,2,\ldots,2g-1$, we see that $A_{j+1}^{\varepsilon}A_{j+2}^{-\varepsilon} = F^{-j} A_1A_2^{-1} F^j$, where $\varepsilon=-1$ if $j$ is odd and $\varepsilon=1$ if $j$ is even. 
This gives 
\begin{align}
A_2^{-1} A_3, \ A_3 A_4^{-1}, \ A_5 A_6^{-1} \in G_3. \label{in2b}
\end{align}
Since $a_2$ is disjoint from $b$, we have $B^{-1}A_2 = A_2B^{-1} = (A_1A_2^{-1})^{-1} A_1B^{-1}$ in $G_3$. 
By (\ref{in2b}), we obtain $A_3^{-1}B = (A_2^{-1}A_3)^{-1} (B^{-1}A_2)^{-1}$ in $G_3$
Since $a_2,a_3$ is disjoint from $b$, we have
\begin{align}
A_2A_3^{-1} = B^{-1}A_2 \cdot (BA_3^{-1}) \in G_3. 
\end{align}
By $(RF)^{-2}(a_2,a_3)=(a_4,a_5)$ and (\ref{in1b}), we get
\begin{align}
A_1 A_2^{-1}, \ A_2 A_3^{-1}, \ A_3 A_4^{-1},\ A_4 A_5^{-1} \ A_5 A_6^{-1} \in G_3. \label{in3b}
\end{align}
Therefore, by $A_jA_{6}^{-1} = A_jA_{j+1}^{-1} \cdots A_4A_5^{-1} \cdot A_5A_6^{-1}$, we have
\begin{align}
A_1 A_6^{-1}, \ A_2 A_6^{-1}, \ A_3 A_6^{-1}, \ A_4 A_6^{-1} \in G_3. \label{in4b}
\end{align}
Since $a_2$ is disjoint from $a_6$, by (\ref{in1b}) and (\ref{in4b}), we obtain
\begin{align}
A_6^{-1}C_1 = A_2A_6^{-1} \cdot A_2^{-1}C_1 \in G_3. \label{in5b}
\end{align}
It is easily seen that 
\begin{align*}
A_6A_4^{-1} \cdot A_6A_3^{-1} \cdot A_6A_2^{-1} \cdot A_6A_1^{-1} (a_6,c_1) = (a_6,\overline{b}). 
\end{align*}
Therefore, by (\ref{in4b}) and (\ref{in5b}) we have 
\begin{align}
A_6^{-1}\overline{B} = (A_6A_4^{-1} \cdot A_6A_3^{-1} \cdot A_6A_2^{-1} \cdot A_6A_1^{-1}) A_6^{-1}C_1 (A_6A_4^{-1} \cdot A_6A_3^{-1} \cdot A_6A_2^{-1} \cdot A_6A_1^{-1})^{-1} \in G_3. \label{in5b}
\end{align}
Since $a_1$ is disjoint from $a_6$, combining (\ref{in4b}) and (\ref{in5b}) we get 
\begin{align*}
A_1^{-1}\overline{B} = (A_1 A_6^{-1})^{-1} A_6^{-1}\overline{B} \in G_3. 
\end{align*}
Since therefore $a_1$ is disjoint from $\overline{b}$, we have
\begin{align}
\overline{B}A_1^{-1}  \in G_3. \label{in6b}
\end{align}

Next, we show that $\overline{E}_{} \ \overline{B}_2^{-1} \in G_3$. 
It is easy to check that 
\begin{align*}
A_6A_1^{-1} \cdot A_5A_1^{-1} \cdot A_4 \overline{C}_2^{-1} (a_6, \overline{b}) = (a_5,\overline{b}_2). 
\end{align*}
Since $A_1A_5^{-1} = A_1A_6^{-1} \cdot (A_5A_6^{-1})^{-1}$ in $G_3$ by (\ref{in3b}) and (\ref{in4b}), we have 
\begin{align*}
A_5^{-1} \overline{B}_2 = (A_6A_1^{-1} \cdot A_5A_1^{-1} \cdot A_4 \overline{C}_2^{-1})  A_6^{-1} \overline{B} (A_6A_1^{-1} \cdot A_5A_1^{-1} \cdot A_4 \overline{C}_2)^{-1} \in G_3
\end{align*}
by (\ref{in1b}) and (\ref{in5b}). 
Here, we have $\overline{B}_2A_2^{-1} = \overline{B}A_1^{-1} \cdot A_1A_2^{-1} \in G_3$ by (\ref{in6b}) and (\ref{in3b}). 
Since $A_3^{-1} A_4 = (RF)^{-1} A_2 A_3^{-1} (RF) \in G_3$ by (\ref{in3b}), we have $C_1^{-1}A_4 = C_1^{-1}A_2 \cdot A_2^{-1}A_3 \cdot A_3^{-1}A_4 \in G_3$ by (\ref{in1b}) and (\ref{in2b}). 
We check at once that 
\begin{align*}
\overline{B}_2A_2^{-1} \cdot \overline{B}_2A_1^{-1} \cdot C_1^{-1}A_4 (a_5, \overline{b}_2) = (\overline{e}, \overline{b}_2). 
\end{align*}
By (\ref{in6b}) we get
\begin{align*}
\overline{E}^{-1}\overline{B}_2 = (\overline{B}_2A_2^{-1} \cdot \overline{B}_2A_1^{-1} \cdot C_1^{-1}A_4) A_5^{-1} \overline{B}_2 (\overline{B}_2A_2^{-1} \cdot \overline{B}A_1^{-1} \cdot C_1^{-1}A_4)^{-1} \in G_3
\end{align*}
Since $\overline{e}$ is disjoint from $\overline{b}_2$, we have
\begin{align*}
\overline{E}_{} \ \overline{B}_2^{-1} \in G_3, 
\end{align*}
which the proof is complete. 
\end{proof}

We prove Theorem~\ref{thm:10} for $p\geq 7$. 
\begin{thm}\label{thm:12}
Set $F:=\sigma_{p-1} \cdots \sigma_2 \sigma_1 E_{p-1} A_{2g} \cdots A_{2} A_1$. 
If $g \geq 3$ and $p \geq 7$, then $\EM$ is generated by $A_2 B^{-1} \sigma_2$ and $RF$. 
\end{thm}
\begin{proof}
Let $G_4$ be the subgroup of $\EM$ generated by $A_2 B^{-1} \sigma_2$ and $RF$. 
We prove that $G_4=\EM$ using Proposition~\ref{prop:4}, that is, $A_1,A_2,\ldots,A_{2g},B,E_{p-1}$ and $\sigma_1$, $\sigma_2,\ldots,\sigma_{p-1}$ are in $G_4$. 
Since $RF$ is orientation reversing diffeomorphism on $\S$, by Lemma~\ref{lem:10} and the exact sequence
\begin{align*}
1 \to \M \to \EM \to \mathbb{Z}/2\mathbb{Z} \to 1, 
\end{align*}
it is sufficient to show that $A_1A_2^{-1}$, $A_1B^{-1}$ and $\sigma_1,\sigma_2,\ldots,\sigma_{p-1}$ are in $G_4$.

Since $(RF)^{-3} (a_2,b,\ell_2) = (a_5, \overline{c}_2,\ell_5)$ by the assumption that $p\geq 7$, we obtain $A_5^{-1} \overline{C}_2 \sigma_5^{-1} = (RF)^{-3} A_2 B^{-1} \sigma_2 (RF)^3 \in G_4$. 
Note that $A_5 \overline{C}_2^{-1} \sigma_5 \in G_4$ since $a_5,\overline{c}_2,\ell_5$ are disjoint from each other. 
Since it is easily seen that 
\begin{align*}
A_5 \overline{C}_2^{-1} \sigma_5 \cdot A_2 B^{-1} \sigma_2 (a_5,\overline{c}_2,\ell_5) = (a_5,b,\ell_5), 
\end{align*}
we have 
\begin{align}
A_5B^{-1}\sigma_5 = (A_5 \overline{C}_2^{-1} \sigma_5 \cdot A_2 B^{-1} \sigma_2) A_5 \overline{C}_2^{-1} \sigma_5 (A_5 \overline{C}_2^{-1} \sigma_5 \cdot A_2 B^{-1} \sigma_2)^{-1} \in G_4. \label{in13b}
\end{align}
In particular, since $\ell_5$ is disjoint from $b,a_5,c_2$, we have 
\begin{align}
B\overline{C}_2^{-1} = \sigma_5^{-1} B A_5^{-1} \cdot A_5 \overline{C}_2^{-1} \sigma_5 \in G_4. \label{in14c}
\end{align}
This gives 
\begin{align}
A_2 \overline{C}_2^{-1} \sigma_2 = A_2 B^{-1} \sigma_2 \cdot B \overline{C}_2^{-1} \in G_4. \label{in14b}
\end{align}
By (\ref{in13b}), we have $A_2 A_5^{-1} \sigma_2 \sigma_5^{-1} = A_2 B^{-1} \sigma_2 \cdot \sigma_5^{-1} B A_5^{-1}$ in $G_4$. 
Since $RF(a_2,a_5,\ell_2,\ell_5) = (a_1,a_4,\ell_1,\ell_4)$, we get 
\begin{align}
A_1^{-1} A_4 \sigma_1^{-1} \sigma_4 = (RF) A_2 A_5^{-1} \sigma_2 \sigma_5^{-1} (RF)^{-1} \in G_4. \label{in15b}
\end{align}
Here, it is immediate that 
\begin{align*}
\sigma_2^{-1} \overline{C}_2 A_2^{-1} \cdot A_1^{-1} A_4 \sigma_1^{-1} \sigma_4 (a_2,\overline{c}_2,\ell_2) = (a_1,\overline{c}_2,\ell_1). 
\end{align*}
By (\ref{in14b}) and (\ref{in15b}), we obtain
\begin{align}
A_1 \overline{C}_2^{-1} \sigma_1 = (\sigma_2^{-1} \overline{C}_2 A_2^{-1} \cdot A_1^{-1} A_4 \sigma_1^{-1} \sigma_4) A_2 \overline{C}_2^{-1} \sigma_2 (\sigma_2^{-1} \overline{C}_2 A_2^{-1} \cdot A_1^{-1} A_4 \sigma_1^{-1} \sigma_4)^{-1} \in G_4. \label{in16b}
\end{align}
Using this and (\ref{in14c}), we get
\begin{align}
A_1B^{-1}\sigma_1 = A_1 \overline{C}_2^{-1} \sigma_1 \cdot \overline{C}_2 B^{-1} \in G_4. \label{in17b}
\end{align}
Since we see at once that 
\begin{align*}
\sigma_1^{-1} B A_1^{-1} \cdot A_1^{-1} A_4 \sigma_1^{-1} \sigma_4  (a_1,b,\ell_1) = (a_1,a_4,\ell_1), 
\end{align*}
by (\ref{in15b}) we have
\begin{align*}
A_1A_4^{-1}\sigma_1 = (\sigma_1^{-1} B A_1^{-1} \cdot A_1^{-1} A_4 \sigma_1^{-1} \sigma_4) A_1B^{-1}\sigma_1 (\sigma_1^{-1} B A_1^{-1} \cdot A_1^{-1} A_4 \sigma_1^{-1} \sigma_4)^{-1} \in G_4. 
\end{align*}
Since $a_1,a_4,\ell_1,\ell_4$ are disjoint from each other, by (\ref{in15b}) we obtain 
\begin{align*}
\sigma_4 = A_1^{-1} A_4 \sigma_1^{-1} \sigma_4 \cdot  A_1A_4^{-1}\sigma_1 \in G_4. 
\end{align*}
This gives $\sigma_1,\sigma_2$ in $G$ since $(RF)^{i}(\ell_4)=\ell_{4-i}$ for $i=1,2,3$. 
By (\ref{in14b}), (\ref{in16b}) and (\ref{in17b}), we have 
\begin{align*}
A_2\overline{C}_2^{-1}, \ A_1\overline{C}_2^{-1}, \ A_1B^{-1} \in G_4, 
\end{align*}
and therefore 
\begin{align*}
A_1A_2^{-1} = A_1\overline{C}_2^{-1} \cdot (A_2\overline{C}_2^{-1})^{-1} \in G_4. 
\end{align*}
Therefore, we see that $A_1,A_2,\ldots,A_{2g},B,E_{p-1}$ are in $G_4$ from Lemma~\ref{lem:1}.

Finally, we show that $\sigma_j \in G_4$ for $j=1,2,\ldots,p-1$. 
From $A_2B^{-1}\sigma_2$ and the definition of $RF$, we see that $\sigma_2$ and $R\sigma_{p-1}\cdots \sigma_2 \sigma_1$ are in $G_4$. 
Since $(R\sigma_{p-1}\cdots \sigma_2 \sigma_1)^{-i}(\ell_2) = \ell_{2+i}^{\pm1}$ for $i=-1,0,1,\ldots,p-2$, we have $\sigma_j \in G_4$ for $j=1,2,\ldots,p-1$.

By Lemma~\ref{lem:10} and Proposition~\ref{prop:4}, we see that $G_4 =\EM$, and the proof is complete. 
\end{proof}

The proof for Theorem~\ref{thm:10} for $0\leq p\leq 6$ is similar to that of Corollary~\ref{cor:1}. 
This proves Theorem~\ref{thm:10}.

\section{Remarks}\label{section5}
In this section, we comment on minimal generating sets for $\M$ and $\EM$ of low genus cases and commutator generating sets for $\M$ and $\EM$. 
Moreover, we state some questions on torsion and involution generating sets for $\M$ and $\EM$.

\subsection{Minimal generating sets for $\M$ and $\EM$ of low genus cases}
We see that $\mathcal{M}_{1,2}$ is generated by two elements $\sigma_1A_1A_2E_1$ and $A_1$. 
Similarly, $\mathcal{M}_{1,2}^\pm$ is generated by two elements $R'A_1A_2E_1$ and $A_1$, where $R'$ is a reflection such that $R'(e_1)=e_1$, $R'(a_i)=a_i$ for $i=1,2$ and $R(x_1)=x_2$. 
The proofs are left to the reader. 
For the reason, we expect that the same holds for $g=1$ and $p\geq 3$ or $g\geq 2$ and $p\geq 2$. 
\begin{ques}
Are $\M$ and $\EM$ generated by two elements if $g=1$ and $p\geq 3$ or $g\geq 2$ and $p\geq 2$?
\end{ques}

\subsection{Commutator generating sets for $\MO$ and $\EMO$}
Powell \cite{Po} showed that $\MO$ is a perfect group if $g \geq 3$, that is, $\MO=[\MO,\MO]$, where $[\MO,\MO]$ is the commutator subgroup of $\MO$. 
It is well-known that the same result holds for $\M$ if $g\geq 3$ and $p=1$. 
Baykur and Korkmaz \cite{BK} showed that $\MO$ is generated by two commutators if $g\geq 5$ and by three commutators if $g\geq 3$. 
However, if $g=1,2$ and $p\geq 0$, then we can not obtain any generating set for $\M$ consisting of commutators since $\M$ is not perfect. 
Similarly, $\M$ can not be generated by commutators for $g\geq 0$ and $p\geq 2$ since there is a surjective homomorphism $\M \to \Sym \to \mathbb{Z}/2\mathbb{Z}$. 
In a similar manner, we see that the same result holds for $\EM$ for $g\geq0$ and $p\geq 0$.

\subsection{Torsion generating sets for $\M$ and $\EM$}
Maclachlan \cite{Mac} proved that $\MO$ is generated by torsion elements, and as a corollary they showed that the moduli space is simply connected as a topological space. 
Luo \cite{Luo} gave a finite generating set for $\M$ consisting of torsion elements for $g\geq 3$ and $p\geq 0$, and also showed that $\M$ is generated by torsions if and only if $(g,p) \neq (2,5k+4)$ for some integer $k$. 
The cardinality of the generating set depends linearly on both $g$ and $p$. 
In the paper, he asked whether there is a universal upper bound, independent of $g$ and $p$, for the number of torsion elements needs to generate $\M$. 
Brendle and Farb \cite{BF} answered Luo's question by constructing a generating set of $\MO$ consisting of three elements of order $2g + 2$, $4g + 2$ and $2$ (or $g$) for $g\geq 1$. 
Korkmaz \cite{Ko1} showed that $\MO$ is generated by two torsion elements, each of order $4g + 2$ for $g\geq 3$ and $p=0,1$, and therefore his torsion generating set is minimal. 
The author \cite{Mo3} proved that $\MO$ is generated by three elements of order three and by four elements of order four for $g\geq3$. 
After that Lanier \cite{La} showed $\MO$ is generated by three elements of order $k\geq 6$ for $g\geq (k-1)^2+1$. 
Yoshihara \cite{Yo} showed that $\MO$ is generated by three elements of order $6$ if $g\geq 10$ and by four elements of order $6$ if $g\geq 5$. 
Yildiz \cite{Yi2} improved the result of \cite{Ko1} by showing that $\MO$ is generated by two elements of order $g$ for $g\geq 6$. 
From the above results, it is natural to ask the following:
\begin{ques}
Is $\M$ generated by two torsion elements for $g\geq 3$ and $p\geq 2$. 
\end{ques}

A minimal torsion generating sets for $\EMO$ were given by Du \cite{Du1,Du2}. 
He showed that $\EMO$ is generated two torsion elements for $g\geq 3$ and that $\mathcal{M}^\pm_1$ cannot be generated by two torsions. 
On the other hand, it is still unknown for $g=2$. 
\begin{ques}
Is $\mathcal{M}^\pm_2$ generated by two torsion elements?
\end{ques}
Altunoz, Pamuk and Yildiz \cite{APY} gave a generating set for $\EM$ consisting three torsion elements of order $g$ for $g\geq 6$ and $p=ng,ng+1,ng+2$, where $n$ is a non-negative integer. 
We do not know whether an analogue of Du's result holds for $\EM$. 
\begin{ques}
Is $\EM$ generated by two torsion elements for $g\geq 3$ for $p\geq 1$. 
\end{ques}

\subsection{Involution generating sets for $\M$}
McCarthy and Papadopoulos \cite{MP} proved that $\MO$ is generated by infinitely many conjugates of a single involution for $g \geq 3$.
Actually, the above mentioned Luo's generating set for $\M$ consists of all involutions for $g\geq 3$. 
Therefore, we can consider the involution version of Luo's question. 
The answer was also given by Brendle and Farb \cite{BF} by showing that $\M$ is generated by six involutions for $g\geq 3$ and $p=0$, and for $g\geq 4$ and $p=1$. 
In the case of $p\geq 2$, Kassabov \cite{Ka} showed that $\M$ is generated by four involutions if $g > 7$ or $g = 7$ and $p$ is even, five involutions if $g > 5$ or $g = 5$ and $p$ is even, and 6 involutions if $g > 3$ or $g = 3$ and $p$ is even. 
The assumption that $p$ is even was removed by the author \cite{Mo2} for $g=7$ and $g=5$. 
Since $\M$ contains nonabelian free groups, it cannot be generated by two involutions. 
Thus, any set of involution generators for $\M$ must have at least three elements. 
Korkmaz \cite{Ko2} proved that $\MO$ is generated by three involutions if $g\geq 8$ and four involutions if $g\geq 3$, answering a question in \cite{BF}. 
Recently, the result was improved to $g\geq 6$ by Yildiz \cite{Yi1}. 
Based on the above works, we would like to ask the following question:
\begin{ques}
Is $\M$ generated by three involutions for $g\geq 3$ and $p\geq 0$. 
\end{ques}
Note that if $g=1,2$, then $\MO$ cannot be generated by involutions for homological reasons. 
Therefore, $\M$ also cannot be generated by involutions for $g=1,2$ since there is a surjective homomorphism from $\M$ to $\MO$ (see \cite{Bi}).

By a similar reason to $\M$, $\EM$ can not generated by two involutions. 
Stukow \cite{St} gave a minimal involution generating set for $\EMO$ for $g\geq 1$, that is, he showed that $\EMO$ is generated by three involutions. 
In \cite{APY}, they also showed that $\EM$ is generated by four involutions for $g\geq 3$ and $p\geq 0$. 
Hence, a natural question arises:
\begin{ques}
Is $\EM$ generated by three involutions for $g\geq 1$ and $p\geq 0$. 
\end{ques}

\end{document}